\documentclass[12pt]{amsart}

\usepackage{amsthm}
\usepackage{amsmath}
\usepackage{amsfonts,amssymb,latexsym}
\usepackage[all,arc,curve,color,frame,line]{xy}
\usepackage[english]{babel}
\usepackage{epsfig,epsf}
\usepackage{color}
\newtheorem{Thm}{Theorem}
\newtheorem{Lemma}{Lemma}
\newtheorem{cor}{Corollary}

\newtheorem{Proposition}{Proposition}

\title[Small doubling in ordered groups]{Small doubling in ordered groups:\\ generators and structure}

\author[Freiman, Herzog, Longobardi, Maj, Plagne \and Stanchescu ]
{G. A. Freiman, M. Herzog, P. Longobardi,\\ M. Maj, A. Plagne \and Y.V. Stanchescu}

\address{Gregory A. Freiman, Marcel Herzog
\newline 
\indent School of Mathematical Sciences, Tel Aviv University,
Tel Aviv 69978, Israel}
\email{~grisha@post.tau.ac.il}
\email{~herzogm@post.tau.ac.il}

\address{Patrizia Longobardi, Mercede Maj
\newline
\indent Dipartimento di Matematica, Universita' di Salerno, Via Giovanni Paolo II, 132, 84084 Fisciano (Salerno), Italy}
\email{~plongobardi@unisa.it}
\email{~mmaj@unisa.it}

\address{Alain Plagne
\newline
\indent Centre de Math\'ematiques Laurent Schwartz, \'Ecole polytechnique, 91128 Palaiseau Cedex,France}
\email{~plagne@math.polytechnique.fr}

\address{Yonutz V. Stanchescu
\newline
\indent Afeka Academic College, Tel Aviv 69107, Israel
\newline 
and
\newline
\indent The Open University of Israel, Raanana 43107, Israel}
\email{~yonis@afeka.ac.il {\rm \it and} ~ionut@openu.ac.il}

\begin{document}

\begin{abstract}
We prove several new results on the structure of the subgroup generated by  
a small doubling subset 
of an ordered group, abelian or not. We obtain precise results generalizing 
Freiman's $3k-3$ and $3k-2$ 
theorems in the integers and several further generalizations.
\end{abstract}

\maketitle


\section{Introduction}

Let $G$ denote an arbitrary group (multiplicatively written). If $S$ is a 
subset of $G$, we define its {\em square} $S^2$ by
$$
S^2=\{x_1x_2\mid x_1,x_2\in S\}.
$$
In the abelian context, $G$ will usually be additively written and we 
shall rather speak of the {\em double} of $S$.

In this paper, we are concerned with the following general problem. 
Suppose that we are given 
two real numbers $\alpha \geq 1$ and $\beta$. While in general, the size 
of $S^2$ should typically be quadratic in the one of $S$, we would like 
to identify those finite sets $S$ having the property 
that
\begin{equation}
\label{star}
|S^2| \leq \alpha|S| + \beta.
\end{equation}
Typically, we may need to add the natural assumption that $| S |$ is 
not too small. It is clear that the difficulty 
in describing $S$ in this problem increases with $\alpha$, 
which is related with the {\em doubling constant} of $S$ defined 
as $| S^2 |/ |S|$. When $| S |$ is large, solving \eqref{star} is 
tantamount to asking for a complete description of sets with a 
bounded doubling constant.

Problems of this kind are called {\em inverse problems} in additive number theory. 
During the last two decades, they became the most central issue in a 
fast growing area, known as additive combinatorics.
Inverse problems of small doubling type have been first investigated 
by G.A. Freiman very precisely in the additive group of the integers 
(see \cite{F1}, \cite{F2}, \cite{F3}, \cite{F4}) and by many other 
authors in general abelian groups, starting with M. Kneser  \cite{K}
(see, for example, \cite{Ham}, \cite{Bilu}, \cite{LS}, \cite{Ru}, \cite{GR}).  
More recently, small doubling problems in non-necessarily abelian groups 
have been also studied, see \cite{G},  \cite{Sa} and \cite{BGT} for 
recent surveys on these problems and  \cite{NA} and \cite{TV} 
for two important books on the subject.

There are two main types of questions one may ask. First, find the general 
type of structure that $S$ can have and how this type of structure behaves 
when $\alpha$ increases. In this case, very powerful general results 
have been obtained (leading to a qualitatively complete structure theorems 
thanks, notably, to the concepts of nilprogressions and approximate groups), 
but they are not very precise quantitatively. 
Second, for a given (in general quite small) range of values for $\alpha$, 
find the precise (and possibly complete) 
description of those sets $S$ which satisfy \eqref{star}. The archetypical 
results in this area are Freiman's $3k-4$, $3k-3$ or $3k-2$ theorems 
in the integers (see \cite{F1}, \cite{F2} and \cite{F3}). See also for 
instance \cite{LS}, \cite{DHP}, \cite{HP}, \cite{HP2}, \cite{S2} or \cite{S3} 
for other results of this type. In this paper, we investigate problems 
of the second type.

Here, for a given non-necessarily abelian group $G$, we would like 
to understand precisely what happens in the case when $\alpha =3$. 
We restrict ourselves to the already quite intricate case of 
{\em ordered groups} started in papers \cite{FHLM}, \cite{FHLMS1} 
and \cite{FHLMS2}.

We recall that if $G$ is a group and $\leq$  is a total order relation defined
on the set $G$, we say that $(G, \leq)$ is {\em an ordered group} if 
for all $a,b,x,y\in G$, the inequality $a\leq b$ implies that $xay\leq xby$. 
A group $G$ is {\em orderable} if there exists an order $\leq$ on $G$ 
such that $(G, \leq)$ is an ordered group.

Obviously the group of  integers with the usual ordering is an ordered group. 
More generally, the class of orderable groups contains all nilpotent 
torsion-free groups (see, for example, \cite{MA} or \cite{N}). We will 
come back to nilpotent groups later in this paper.

In \cite{FHLM}, some of us proved that if  
$S$ is a finite subset of an orderable group satisfying
$|S| \geq 3$ and $|S^2| \leq 3|S| - 3$,
then $\langle S \rangle$, the subgroup generated by $S$, is abelian. 
Moreover, if $|S| \geq 3$ and 
$|S^2|  \leq 3| S | -4$, then
$S$ is a subset of a short commutative {\em geometric progression}, 
that is, a set of the form
$$
P_l (u,t) = \{ u, ut,...., ut^{l-1} \} \subseteq \langle u,t \rangle
$$
for some {\em commuting} elements $u$ and $t$ of $G$ and $l$ an integer ($t$ will
be called the {\em ratio} of the progression and $l$ its {\em length}).
In particular, $\langle S \rangle$ is abelian and at most {\em 2-generated}
(by which we mean that it is generated by a set having at most $2$ elements). 
Finally, a group was constructed such that for any integer $k \geq 3$ there exists 
a subset $S$ of cardinality $k$ such that 
$|S^2| = 3|S|-2$ and $\langle S \rangle$ is non-abelian.

In this paper, we continue to study the structure of finite subsets $S$ 
of an orderable group when an inequality of the form \eqref{star} 
with $\alpha=3$ holds and improve drastically our preceding results, 
both by extending their range of application and by making them more 
precise. We particularly concentrate on the size of a generating set 
and the structure of $\langle S \rangle$, as will be explained in 
the following section.


\section{New results and the plan of this paper}

The present paper is organized as follows.

In Section \ref{sec2}, we consider the case when $\langle S \rangle $ is abelian.
We first obtain a $(3k-3)$-type theorem generalizing Freiman's classical theorem 
in the integers.

\begin{Thm}[Ordered $3k-3$ Theorem]
\label{3k-3}
Let $G$ be an orderable group and $S$ be a finite subset of $G$ satisfying
$|S^2| \leq 3 |S|-3$. Then $\langle S \rangle$ is abelian
and at most $3$-generated. 

Moreover, 
if $|S| \geq 11$, then one of the following two possibilities occurs:
\begin{itemize}
\item[(i)] $S$ is a subset of a geometric progression of length at most $2|S|-1$,
\item[(ii)] $S$ is the union of two geometric progressions with the same ratio, 
such that the sum of their lengths is equal to $|S|$.
\end{itemize}
\end{Thm}

In general, this result cannot be improved. However, it will follow from our 
proof that a slightly more detailed result 
(namely, describing those sets $S$ with $|S| \leq 10$ involved in this statement) 
can be easily achieved if one employs 
more carefully the tools used in our proof. It will be shortly explained 
how to do this at the end of our proof.
For the sake of simplicity and avoiding technicalities, here we prefer 
a cleaner statement. 

We can go a step further and formulate also a precise generalized 
$(3k-2)$-type result.

\begin{Thm}[Abelian ordered $3k-2$ Theorem]
\label{3k-2}
Let $G$ be an orderable group and $S$ be a finite subset of $G$ such that  
$|S^2| =  3|S|-2$ and $\langle S \rangle$ is abelian. Then either 
$|S|=4$ or $\langle S \rangle$ is at most 3-generated.

Moreover, if $|S| \geq 12$, then one of the following two possibilities occurs:
\begin{itemize}
\item[(i)] $S$ is a subset of a geometric progression of length at most $2|S|+1$,
\item[(ii)] $S$ is contained in the union of two geometric progressions with 
the same ratio, such that the sum of their lengths is equal to $|S|+1$.
\end{itemize}
\end{Thm}

Again, this result is in general best possible and, again, at the price 
of increasing the number of
sporadic cases and doing a more careful examination of our proofs, 
one could replace the assumption $|S| \geq 12$ in Theorem \ref{3k-2} 
by a more precise list of possible situations. For the sake of 
clarity, we prefer such a neat statement.

We then investigate the general abelian case. We obtain the following
result.

\begin{Thm}
\label{ck}
Let $c$ be an integer, $c \geq 2$.
Let $G$ be an orderable group and $S$ be a finite subset of $G$ such that 
$\langle S \rangle$ is abelian. If 
$$
|S^2| < (c+1) |S| - \frac{c(c+1)}2,
$$ 
then  $\langle S \rangle$ is at most $c$-generated.
\end{Thm}
We obtain for instance the following corollary, corresponding to 
the case $c=3$ above.

\begin{cor}
\label{coro4k-6}
Let $G$ be an ordered group and $S$ be a finite subset of $G$ such 
that $\langle S \rangle$ is abelian. 
If $|S^2| \leq 4|S|-7$, then $\langle S \rangle$ is at most $3$-generated.
\end{cor}

If needed, we could even make this result more precise and give the 
precise structure of $S$ by using the results of 
\cite{S1}, \cite{S2} and \cite{S3}.

In Section \ref{sec3}, we go back to the general case of non-necessarily 
abelian groups and 
study the maximal number of generators of $\langle S \rangle$. 
When merged with some of our preceding results, we obtain the following 
general result.

\begin{Thm}
\label{thmgeneral}
Let $G$ be an ordered group and $S$ be a finite subset of $G$. Then the 
following statements hold:
\begin{itemize}
\item[(i)] $|S^2|\geq 2|S|-1$,
\item[(ii)] If $2|S|-1\leq |S^2|\leq 3|S|-4$, then $\langle S \rangle$ 
is abelian and at most $2$-generated,
\item[(iii)] If $|S^2|=3|S|-3$, then $\langle S \rangle$ 
is abelian and at most $3$-generated,
\item[(iv)] Let $|S^2|= 3|S|-3+b$ for some integer
$b \geq 1$. 		
Then either $|S|=4$, $b=1$, $\langle S \rangle$ is abelian and at 
most $(b+3)$-generated
or $\langle S \rangle$ is at most $(b+2)$-generated.
\end{itemize}
\end{Thm}

In Sections \ref{sec4card3}, \ref{sec4lemgen}, \ref{sec4card4} and \ref{sec4}, 
we look for a complete description 
of $\langle S \rangle$, if $S$ is a finite subset of cardinality $\geq 4$ 
of an orderable group and $|S^2| = 3|S|-2$. The flavour of these results 
is more group-theoretical.
Here we briefly recall for completeness sake a few standard notations 
that we shall need. 
We refer to \cite{Robinson} for all the group-theoretic complementary notation 
that the reader may need for reading this paper.

If $a$ and $b$ are elements of a group $G$, their {\em commutator} is denoted by
$$
[a, b] =a^{-1}b^{-1}ab 
$$
and the {\em derived subgroup} of $G$, denoted by $G'$, is simply the subgroup 
of $G$ generated by its commutators. 
More generally, if $H$ and $K$ are two subgroups of $G$, we denote by $[H,K]$ 
the subgroup generated by the commutators 
of the form $[h,k]$ for $h$ in $H$ and $k$ in $K$. With this notation 
$$
G'=[G,G].
$$
We shall also use the classical notation
$$
a^b=b^{-1}ab.
$$
If $X \subseteq G$, we use the notation $C_G (X)$  for the {\em centralizer} 
of $X$ in $G$ defined as the subgroup 
of elements of $G$ commuting with all the elements of $X$ and  
$$
Z(G)=C_G(G)
$$
is the {\em center} of $G$. If $H$ is a subgroup of $G$, then the 
{\em normalizer} of $H$ in $G$ is by definition the subgroup 
$$
N_G(H) = \{ g \in G \ | \ g^{-1}Hg = H \},
$$ 
and it is the largest subgroup of $G$ containing $H$ in which $H$ is normal. 
Recall finally that,  if $n$ is a positive integer, 
a {\em soluble group of length at most n} is a group which has an 
{\em abelian series} of length $n$ that is, a finite chain 
\begin{equation}
\label{chain}
\{ 1 \} = G_0 \leq G_1 \leq \cdots \leq G_n = G
\end{equation}
of subgroups, such that $G_i$ is a normal subgroup of $G_{i+1}$ and 
$G_{i+1}/G_i$ is abelian for any index $0 \leq i \leq n-1$; 
and  a {\em nilpotent group of class at most n} is a group which has a 
{\em central series} of length $n$ of the form \eqref{chain} such that
$G_{i+1}/G_i$ is contained in the center of $G/G_i$, for any index 
$0 \leq i \leq n-1$. 
Soluble groups of length at most 2 are also called {\em metabelian}; 
they are exactly the groups $G$ such that $G'$ 
is abelian. Nilpotent groups  of class at most 2 are exactly the groups
$G$ such that the derived subgroup $G'$ is contained 
in the center $Z(G)$.

In order to state our $(3k-2)$-type result, we introduce a definition. 
In this paper, we shall say that an ordered group $G$ is 
{\em young} if one of the following occurs:
\begin{itemize}
\item[(i)] $G = \langle a, b  \rangle$ with $[[a, b],a]=[[a, b],b]=1$,	
\item[(ii)] $G = \langle a \rangle \times \langle b, c  \rangle$ with 
either $c^b = c^2$ or $(c^2)^b=c$ and $c\neq 1$,
\item[(iii)] $G = \langle a, b \rangle$ with  $a^b = a^2$ and $a\neq 1$,
\item[(iv)] $G = \langle a, b \rangle$ with $a^{b^2} = aa^b$, $[a, a^b] = 1$
and $a\neq 1$.
\end{itemize}
More precisely, we shall speak of an ordered young group of type (j) if the group 
satisfies the definition (j) in the list above (j being i, ii, iii or iv).
Notice that an ordered young group of type (iii) is a quotient of $B(1,2)$, 
the Baumslag-Solitar group \cite{BS}. 

Notice also that nilpotent ordered young groups are of type (i). This follows
from the fact that in nilpotent groups the derived subgroup is contained in the set 
of all non-generators (see for example \cite{R}, Lemma 2.22). 
Consequently, nilpotent ordered young groups are of class at most 2. 

Moreover, we claim the following.

\begin{Lemma}
\label{young}
An ordered young group is metabelian and a nilpotent ordered young
group is of nilpotency class at most $2$.
\end{Lemma}

\begin{proof}
This is obvious in case (i), since in this case $G$ is nilpotent of class at most 2. 

So suppose that either (ii) or (iii) or (iv) holds. 
Write $H = < c^{b^i} | i \in  \mathbb{Z}>$ in case (ii) and 
$H = <a^{b^i} | i \in \mathbb{Z}>$ in cases (iii) and (iv). 
Then $H$ is normal in $G$ and $G/H$ is abelian. Thus $G^{\prime}$ is 
contained in $H$, and since the converse is also true, it follows
that $G^{\prime} = H$. By induction on $n$ it is easy to prove 
that for every $n \in \mathbb{N}$,  $[c, c^{b^n}] = 1$ in case (ii) 
and $[a, a^{b^n}] = 1$ in cases (iii) and (iv). This result also implies that
$[c^{b^{-n}}, c] = 1$ in case (ii) and  $[a^{b^{-n}}, a] = 1$  in 
cases (iii) and (iv). Thus $c \in Z(H)$ in case (ii) and $a \in Z(H)$ 
in the other cases,  which implies that $H \leqslant Z(H)$. Hence $H$ is abelian
and $G$ is metabelian, as claimed.
\end{proof}

Our main result in this paper is the following theorem.

\begin{Thm} 
\label{3k-2general}
Let $G$ be an ordered group and let $S$ be a finite subset of $G$. 
If $|S| \geq 4$ and $|S^2| = 3|S|-2$, then $ \langle S \rangle$ is 
either abelian or young.
\end{Thm}

Theorem \ref{3k-3}, Theorem \ref{3k-2general} and Lemma \ref{young}
yield the following two corollaries.

\begin{cor}
\label{cor2}
Let $G$ be an ordered group and let $S$ be a finite subset of $G$.
If $|S| \geq 4$ and $|S^2|\leq 3|S|-2$, then $\langle S \rangle$ is metabelian.
\end{cor}

\begin{cor}
\label{cor1}
Let $G$ be an ordered group and let $S$ be a finite subset of $G$ such that
$\langle S \rangle$ is nilpotent.
If $|S| \geq 4$ and $|S^2| \leq 3|S|-2$, then $\langle S \rangle$ is
of nilpotency class at most $2$.
\end{cor}

Our preceding results show that if $|S| \geq 3$ and $|S^2| \leq 3|S|-3$, 
then $\langle S \rangle$ is abelian, 
but, for any integer $k \geq 3$, there exist an ordered group $G$ and a 
subset $S$ of $G$ of size $k$ 
with $|S^2| = 3|S|-2$ and $\langle S \rangle$ non-abelian \cite{FHLM}.
Moreover, we have proved in Corollary \ref{cor2} that if $|S| \geq 4$ and 
$|S^2| \leq 3|S| -2$, then $\langle S \rangle$ is metabelian. 
It is now natural to ask whether there exist an ordered group $G$ and 
a positive integer $b$, such that for any integer $k$ there is 
a subset $S$ of $G$ of 
order $k$ with $|S^2| = 3|S|-2+b$ and $\langle S \rangle$ non-metabelian 
or more generally non-soluble. 
We give a negative answer to this question in Section \ref{sec5} 
by proving the following result.

\begin{Thm}
\label{3k-2+s}
Let $G$ be an ordered group and $s$ be any positive integer.  
If $S$ is a subset of $G$ of cardinality $\geq 2^{s+2}$ such that 
$|S^2| = 3|S|-2+s$, then $\langle S \rangle$ is metabelian.
Moreover, if $G$ is nilpotent, then  $\langle S \rangle$ is
nilpotent of
class at most $2$.
\end{Thm}

Our final result, Theorem \ref{construction}, is a constructive result.

\begin{Thm}
\label{construction}
There exists a non-soluble ordered group $G$ such that for any integer 
$k \geq 3$, there 
exists a subset $S$ of $G$ 
of cardinality $k$ such that $|S^2| = 4|S|-5$ and $\langle S \rangle=G$.
In particular, $\langle S \rangle$ is non-soluble.
\end{Thm}


\section{The abelian case : proofs of Theorems \ref{3k-3}, \ref{3k-2} 
and \ref{ck}} 
\label{sec2}

Let $S$ be a finite subset of an ordered group $G$ and suppose that
$\langle S \rangle$ is abelian. Since $\langle S \rangle$ is finitely 
generated and ordered (thus torsion-free), it is 
isomorphic to some $({\mathbb Z}^m, + )$, for $m=m(S)$, an integer. 
In other words, the additive group $\langle S \rangle$ is 
an {\it $m$-generated group}.
We may thus make our reasoning in this setting, which is simpler 
and was already studied. 
Notice that in this section we shall always use the additive notation.

The notion of {\em Freiman dimension} 
of the set $S$ will be needed here \cite{F3}. 
Recall first that $S$ is {\em Freiman isomorphic} to a set 
$A\subseteq {\mathbb Z}^d$
if there exists a bijective mapping $F:S\to A$ 
such that the equations $g_1+g_2=g_3+g_4$ and $F(g_1) +F(g_2)=F(g_3) +F(g_4)$
are equivalent for all $g_1,g_2,g_3,g_4 \in S$ (i.e. sums of two elements
of $S$ coincide if and only if their images have the same property).

The {\em Freiman dimension} of $S$ 
is then defined as 
the largest integer $d=d(S)$ such that $S$ is Freiman-isomorphic to a subset $A$
of ${\mathbb Z}^d$ not contained 
in an affine hyperplane of ${\mathbb Z}^d$ (i.e. $A$ is not contained in
a subset of ${\mathbb Z}^d$ of the form $L=a+W$, where $a\in {\mathbb Z}^d$
and 
$W=\mathbb Z\omega_1\oplus\dots\oplus\mathbb Z\omega_{d-1}$).

We recall immediately the basic inequalities linking $m(S), d(S)$ and $|S|$ for a subset $S$ of some 
ordered abelian group.

\begin{Lemma}
\label{remark}
Let $S$ be a finite subset of an ordered group $G$.
Assume that $\langle S \rangle$ is abelian.
Then
$$
m(S) \le d(S)+1 \leq |S|.
$$
\end{Lemma}

Both inequalities are tight as shown by the example
$$
S=\{(0,1),(1,1)\}\subseteq {\mathbb Z}^2.
$$
Indeed, $d(S)=1$,
while $\langle S \rangle={\mathbb Z}^2$ and $m(S)=2=|S|$.

\begin{proof}
By our assumptions,
the additive group $\langle S \rangle $ is
isomorphic to $({\mathbb Z}^m, + )$, for $m=m(S)$, an integer.
Without loss of generality we may assume that
$$S \subseteq {\mathbb Z}^m\quad \text{and}\quad\langle S \rangle = {\mathbb Z}^m.$$

Let $r$ be the smallest dimension of a subspace
$W ={\mathbb Z}w_1\oplus...\oplus {\mathbb Z}w_r\subseteq {\mathbb Z}^m$ such that
$$L=a+W= a+{\mathbb Z}w_1\oplus...\oplus {\mathbb Z}w_r$$
contains $S$ for some $a\in {\mathbb Z}^m$. Then
$$r\leq m=m(S)\leq r+1,$$
which implies that either $r=m-1$ or $r=m$.

If $r=m-1$, then 
there is a point $a \in  {\mathbb Z}^m$ and a subspace 
$W={\mathbb Z}w_1\oplus...\oplus {\mathbb Z}w_{m-1}\subseteq {\mathbb Z}^m$
such that 
$$S \subseteq a+W.$$
The set $S^* = S-a=\{s-a\mid s\in S\}$ 
is Freiman isomorphic to $S$ by the trivial correspondence:
$s \leftrightarrow s-a$ for all $s\in S$ and it is contained in 
$W={\mathbb Z}w_1\oplus...\oplus {\mathbb Z}w_{m-1}.$
Let $$T:W \rightarrow {\mathbb Z}^{m-1}$$ be the unique linear isomorphism defined by 
$$T(w_1)=e_1=(1,0,...,0),...,T(w_{m-1})=e_{m-1}=(0,...,0,1).$$
Define $A=T(S^*)$.
Using the hypothesis $r = m-1$, it follows that the set $A$ is not contained 
in an affine hyperplane of ${\mathbb Z}^{m-1}$. Therefore
$d(A)\ge m-1$. The sets $A$, $S^*$ and $S$ are Freiman isomorphic to each other
and thus 
$$d(S)=d(S^*)=d(A)\ge m-1.$$

If $r=m$, then 
the set $S$ is not contained in an affine hyperplane of ${\mathbb Z}^m.$
Thus $d(S) \ge m > m-1.$
Hence $m\leq d(S)+1$ in all cases.

We prove now the inequality $ d +1 \le |S|$. Let $|S|=k$. 
The set $S$ is Freiman-isomorphic to
a subset
$A =\{a_0,a_1,...,a_{k-1}\} \subseteq {\mathbb Z}^d$ not contained
in an affine hyperplane of ${\mathbb Z}^d $. Without loss of generality,
we may assume that $a_0=0$ (indeed, we may replace $A$ by the  
translate $A-a_0$
and use
$d(A)=d(A-a_0)$ and $|A|=|A-a_0|$).
The set $A$ is contained in the subspace
$$W={\mathbb Z}a_0+{\mathbb Z}a_1+...+{\mathbb Z}a_{k-1}={\mathbb
Z}a_1+...+{\mathbb Z}a_{k-1}.$$
If $k \le d$, then $m(W) \le k-1 \le d-1$ and this contradicts our hypothesis
that $A$ is not contained in an affine hyperplane of ${\mathbb Z}^d $.
Hence $k=|A|=|S| \ge d+1$ as required. The proof of the lemma is complete.
\end{proof}
 
For the proof of the Theorems we shall use  Freiman's
theorem (Lemma 1.14 of \cite{F3}), stating that 
for any finite subset $S$ of a torsion free group with $\langle S \rangle$ abelian 
and with Freiman dimension $d$, 
the following lower bound 
for the cardinality of $S^2$ holds:
\begin{equation}
\label{freiman}
| S^2 | \geq (d+1)|S| - \frac{d(d+1)}2.
\end{equation}
This inequality was actually proved in Lemma 1.14 of \cite{F3} for sets with 
affine dimension $d$. However, if $S$ has Freiman dimension $d$, then $S$
is  Freiman isomorphic to some $S^{*}\subseteq  {\mathbb Z}^d $ which has 
affine dimension $d$ and satisfies the equations $|S^{*}|=|S|$ and 
$|(S^{*})^2|=|S^2|$. Therefore 
$$|S^2|=|(S^{*})^2|\geq (d+1)|S^{*}|-\frac {(d+1)d}2=(d+1)|S|-\frac {(d+1)d}2.$$

The proof of the Theorems can now be obtained easily.
Notice that  by Theorem 1.3 
in \cite{FHLM}, also the set $S$ of Theorem \ref{3k-3}
generates an abelian group. 

\begin{proof}[Proof of Theorems \ref{3k-3} and \ref{3k-2}, the common part]
Let $c=2$ or $3$. By our assumptions and \eqref{freiman} we obtain
$$
3 |S| -c \ge  | S^2 | \geq (d+1)|S| - \frac{d(d+1)}2,
$$
where $d=d(S)$. This yields
\begin{equation}
\label{precise}
(d-2)|S| \leq \frac{d(d+1)}2 - c.
\end{equation}
Then, using Lemma \ref{remark}, we obtain
$$
(d-2)(d+1) \leq \frac{d(d+1)}2 - c 
$$
and therefore
\begin{equation}
\label{commut}
(d-1)(d-2) \leq 2(3-c).
\end{equation}
\end{proof}

Now, we have to continue separately the study of the cases $c=2$ and $3$.

\begin{proof}[Proof of Theorem \ref{3k-3}, concluded] 
Here $c=3$, and as mentioned above,  $\langle S \rangle$ is abelian.
Thus \eqref{commut} implies that either $d=1$ or $d=2$. 
Hence by Lemma \ref{remark} $m(S)\leq 3$ 
and $\langle S \rangle$ is at most $3$-generated, as required.

Suppose now that $|S|\geq 11$. Since either $d=1$ or $d=2$, we are back to
the case of the $1$- or the $2$-dimensional $3k-3$ theorem in the integers. 
If $d=1$, such sets are described in Theorem 1.11 in \cite{F3} (or see \cite{HP}): 
they are subsets of an arithmetic progression of length $2|S|-1$ 
or the union of two arithmetic progressions with the same ratio 
(a remaining case corresponds 
to bounded cardinality). 
If $d=2$, such sets are described by Theorem 1.17 in \cite{F3} 
(or see Theorem B in \cite{S2}): 
they are the union of two arithmetic progressions with the same ratio
(there are also remaining cases of bounded cardinality). 
 
This is enough to prove Theorem \ref{3k-3}.
If one wants to be more precise, it is enough to use the same Theorems 1.11 and 1.17 of 
\cite{F3} quoted above, but in their precise forms.
\end{proof}

\begin{proof}[Proof of Theorem \ref{3k-2}, concluded]
Here $c=2$, thus the bound \eqref{commut} implies that $d=1, 2$ or $3$. If $d=3$, then
by \eqref{precise}, we obtain that $|S| \leq 4$ and if $d=1,2$, then  
$\langle S \rangle$ is at most $3$-generated by Lemma \ref{remark}. Hence either $|S|=4$ or 
$\langle S \rangle$ is at most $3$-generated, as required.

If $d=1$, we are back to the $3k-2$ theorem in the integers. Hence it 
follows by Theorem 1.13 of \cite{F3} (or see the original publication \cite{F2}; we notice that 
in the original statement of this theorem, there is a missing sporadic case of size $11$ which 
makes it necessary to have $12$ here instead of $11$) that if $|S| \geq 12$ , 
then $S$ must be either a subset of a short 
geometric progression or a geometric progression minus its second element, 
together with an isolated point, or, finally, a geometric progression together
with another geometric progression of length $3$ with the same ratio and with 
the middle term missing. The last possibility is missing in Theorem 1.13, as
printed in \cite{F3}. Thus $S$ satisfies either (i) or (ii).

If $d=2$, then the set is of Freiman dimension $2$ and its structure 
is described  in Theorem 1.17 of \cite{F3} (or see Theorem B  in \cite{S2}). 
This case leads to possibility (ii) in the statement of Theorem \ref{3k-2}. 
Again, if one wishes to deal with small cardinalities for $S$, 
one simply has to use more precise versions of Freiman's theorems. 
\end{proof}

The proof of Theorem \ref{ck} follows the same lines.

\begin{proof}[Proof of Theorem \ref{ck}]
As above, if $d=d(S)$, we must have
$$
(c+1)|S|- \frac{c(c+1)}2 > | S^2 | \geq (d+1)|S| - \frac{d(d+1)}2.
$$
This implies by factorization that
$$
(d-c)|S| < \frac{(c-d)(c+d+1)}{2}			
$$
from which it follows that one cannot have $d \geq c$. Thus $d \leq c-1$ and the conclusion follows by Lemma \ref{remark}.
\end{proof}


\section{On the generators of $\langle S \rangle$: proof of Theorem \ref{thmgeneral}}
\label{sec3}

We now come back to the general case of non-necessarily abelian groups and start with a lemma.

\begin{Lemma}
\label{lemma1}
Let $G$ be an ordered group. Let $T$ be a non-empty finite subset of $G$ and let $t$ denote its 
maximal element $\max T$. Let $x \in G$ be an element satisfying $x> t$. If 
$yx, xy  \in T^2$ for each $y \in T (\setminus \{  t \})$, then
$T \subseteq  \langle t, x\rangle$. 
\end{Lemma}

\begin{proof}
Write $l=|T|$ and 
$$
T = \{t_l, t_{l-1}, \cdots, t_1\}, \quad t_l < t_{l-1} < \cdots < t_1=t.
$$

We prove by induction on $j$ ($1 \leq j \leq l$) that $t_j \in  \langle t, x\rangle$. 

If $j=1$, we have $t_1=t \in \langle t, x \rangle$. 

Assume that we have proved that $t_1,\dots, t_j \in \langle t, x \rangle$
for some $j$ satisfying $1\leq j \leq l-1$.
Without loss of generality, we may assume that $x t_{j+1} \leq t_{j+1} x$. 

By assumption, we have $x t_{j+1}, t_{j+1} x \in T^2$. It follows that 
$t_{j+1} x = t_u t_v$ for some $t_u,t_v\in T$. But $x>t \geq  t_u,t_v$
and $t_{j+1} x = t_u t_v\geq x t_{j+1}$, so  
the ordering implies that $t_u,t_v >t_{j+1}$.
Therefore
$u,v \leq j$, and the induction hypothesis implies that 
$t_u, t_v \in \langle t, x \rangle$.
Consequently  $t_{j+1}= t_u t_v x^{-1} \in \langle  t, x \rangle$ and
the inductive step is completed. Hence $T \subseteq  \langle t, x\rangle$, as required.
\end{proof}

We pass now to the proof of Theorem \ref{thmgeneral}.

\begin{proof}[Proof of Theorem \ref{thmgeneral}]
For parts (i), and (ii), see \cite{FHLM}, Theorem 1.1 and Corollary 1.4,
respectively. Assertion (iii) follows by our Theorem \ref{3k-3}.

Our aim now is to prove part (iv). Suppose that $|S|=k$ and 
$S=\{x_1<x_2<\dots<x_k\}$.  We argue by induction on $k$.  
If $|S|=k\leq b+2$, then the result is trivial. So suppose that $k\geq b+3$ 
and that part (iv) of Theorem \ref{thmgeneral}  is true up to $k-1$.

Assume, first, that $\langle S \rangle$ is abelian. 
If $b=1$, then
$|S^2|=3k-2$ and by Theorem \ref{3k-2} either $|S|=4$ or $\langle S \rangle$  
is at most $3$-generated, as required since $b+2=3$. So assume that $b\geq 2$. 
Since $k\geq b+3\geq 5$, we have $b\leq k-3$ and 
$$|S^2|\leq 4k-6< 5k-10.$$ But $5k-10=(4+1)k-4(4+1)/2$,
so by Theorem \ref{ck}  $\langle S \rangle$ is at most $4$-generated, as
required since $4\leq b+2$. This conludes the proof of the abelian case.

So assume that  $\langle S \rangle$ is non-abelian and let $T=\{x_1,\dots,x_{k-1}\}$. 
If $x_ix_k,x_kx_i\in T^2$
for all $i<k-1$,  then it follows by Lemma
\ref{lemma1} that $T\leq \langle x_{k-1},x_k \rangle$ and hence  $\langle S
\rangle=
\langle x_{k-1},x_k \rangle$ is at most $2$-generated, as required since
$2<b+2$.  

Hence we may assume that $x_ix_k $ (or $x_kx_i)\notin T^2$ for some $i<k-1$.
Then, by the ordering in $G$, we have $x_ix_k,x_{k-1}x_k,x_k^2\notin T^2$,
which implies that
$$|T^2|\leq |S^2|-3=3k-3+b-3=3(k-1)-3+b.$$
Applying induction, we may conclude that either $\langle T \rangle$ is abelian
and $|T|=4$, or  $\langle T \rangle$ is at most $(b+2)$-generated.

If $x_k\in  \langle T \rangle$, then  $\langle T \rangle= \langle S \rangle$
and $\langle T \rangle$ is non-abelian. Hence
$\langle T \rangle$ is at most $(b+2)$-generated and so is $\langle S \rangle$,
as required.

So assume that $x_k\notin \langle T \rangle$. Then $x_kx_i,x_ix_k\notin
T^2$ for all $i\leq k-1$ and $x_k^2\notin T^2$ by the ordering in $G$.
Thus
$$S^2=T^2\dot\cup (x_kT \cup Tx_k) \dot\cup \{x_k^2\} $$
and since $|x_kT \cup Tx_k|\geq k-1$, it follows that
$$|T^2|\leq |S^2|-k\leq 3k-3+b-(b+3)=3(k-1)-3.$$
Hence, by Theorem 1.3 in \cite{FHLM}, $\langle T \rangle$ is abelian
and since $\langle S \rangle$ is non-abelian, it follows that $x_k\notin C_G(T)$.
Consequently, by Corollary 1.4 in \cite{FHLM}, $|x_kT \cup Tx_k|\geq k$,
which implies that 
$$|T^2|\leq |S^2|-(k+1)\leq 3k-3+b-(b+4)= 3(k-1)-4.$$
Hence, by Proposition 3.1 in \cite{FHLM},  $\langle T \rangle$ is at most
$2$-generated, so $\langle S \rangle$ is at most $3$-generated, as required
since $3\leq b+2$.

The proof of Theorem \ref{thmgeneral}  is now complete.
\end{proof}


\section{The structure of $\langle S \rangle$ if $|S^2| = 3|S|-2$: cardinality 3}
\label{sec4card3}

In Theorem \ref{3k-2general} we assume that $|S|\geq 4$. However, for the
inductive proof of that theorem, we need some 
special results concerning the case when $|S|=3$. These results are proved in the
next two propositions.

\begin{Proposition}
\label{prop3zs}
Let $G$ be an ordered group. Let $x_1, x_2, x_3$ be elements of $G$, 
such that $x_1 < x_2 < x_3$ and let $S = \{x_1, x_2, x_3\}$.
Assume that $\langle S \rangle$ is non-abelian and either $x_1x_2 = x_2x_1$ 
or $x_2x_3 = x_3x_2$.  

Then $|S^2| = 7$ if and only if one of the following holds:
\begin{itemize}
\item[(i)] $S \cap Z(\langle S \rangle) \not= \emptyset$,\ or
\item[(ii)] $S$ is of the form $\{a, a^b, b\}$, where $aa^b = a^ba$.
\end{itemize}
\end{Proposition}

\begin{proof} 
Suppose that $|S^2| = 7$ . Assume first that
$x_1x_2 = x_2x_1$. The other case ($x_2x_3 = x_3x_2$) follows by reversing the
ordering in $G$.

If either $x_2x_3 = x_3x_2$ or $x_3x_1 = x_1x_3$, then either
$x_2 \in Z(\langle S \rangle)$ or $x_1 \in Z(\langle S \rangle)$,  
respectively, and (i) holds. 

So suppose that $x_2x_3 \not = x_3x_2$ and  $x_3x_1 \not = x_1x_3$. It
follows that  $x_3 \notin \langle x_1, x_2 \rangle$. 

If $x_1x_3 < x_3x_1$, then $x_1x_3 \not= x_2x_3$, $x_1x_3 
\not= x_3x_2$, $x_1x_3 \not = x_3^2$ and 
$$
S^2 = \{x_1^2, x_1x_2, x_2^2, x_2x_3, x_3x_2, x_3^2, x_1x_3\}.
$$ 
Hence $x_3x_1 = x_2x_3$ and  $x_1 = x_2^{x_3}$. Thus (ii) follows by taking 
$a=x_2$ and $b=x_3$, since then $x_1 = a^b$. 
Similar arguments yield also (ii) if $x_1x_3>x_3x_1$.

Conversely, if either (i) or (ii) holds, then it is easy to verify that $|S^2| = 7$.

\end{proof}

\begin{Proposition}
\label{prop5}
Let $G$ be an ordered group. Let $x_1, x_2, x_3$ be elements of $G$, 
such that $x_1 < x_2 < x_3$ and 
let $S = \{x_1, x_2, x_3\}$. Assume that both $x_1x_2 \not= x_2x_1$ 
and $x_2x_3 \not = x_3x_2$ (in particular $\langle S \rangle$ is non abelian)
and $|S^2| = 7$.  Then $S$ is of one of the following forms:
\begin{itemize}
\item[(a)] Either $\{x, xc, xc^x\}$ or $\{x^{-1}, x^{-1}c, x^{-1}c^x\}$ 
for some  $c \in G^{\prime}$ satisfying $c > 1$, with $c^{x^2} = cc^x$ 
and $cc^x = c^xc$,
\item[(b)] either $ \{x, xc, xcc^x\}$ or $\{x^{-1}, x^{-1}c, x^{-1}cc^x\}$ 
for some $c \in G^{\prime}$ satisfying $c > 1$, with $c^{x^2} = cc^x$ 
and $cc^x = c^xc$,
\item[(c)] $\{x, xc, xc^2\}$ for some $c \in G^{\prime}$ satisfying $c > 1$, 
with either $c^x = c^2$ or $(c^2)^x = c$.
\end{itemize}

Moreover, in cases (a) and (b),  one has
$$
\langle S \rangle = \langle a, b \rangle  
\quad \text{ with } \quad a^{b^2} = aa^b\quad \text{and} \quad aa^b=a^ba
$$
and $\langle S \rangle$ is young of type (iv), while in case (c)
$$
\langle S \rangle = \langle a, b \rangle, \quad \text{ with } \quad a^b = a^2
$$ 
and $\langle S \rangle$ is young of type (iii).
\end{Proposition}

\begin{proof} 

Write $T = \{x_1, x_2\}$, then $S^2 = T^2 \dot\cup \{x_2x_3, x_3x_2, x_3^2\}$. 

First suppose  that $x_1x_3 \leq x_3x_1$. Then 
$x_1x_3 \notin \{x_2x_3, x_3x_2, x_3^2\}$, and 
we have either $x_1x_3 = x_2^2$ or $x_1x_3 = x_2x_1$. We distinguish 
between three cases.
\medskip

\noindent Case 1:  $x_1x_3 = x_2^2$. 

In this case $x_1x_3 < x_3x_1$, since otherwise $x_1 \in C_G(x_3)$, 
$[x_1, x_2^2] = 1$ and then $[x_1, x_2] = 1$, a contradiction. 
Hence 
\begin{equation}
\label{eq31}
x_3x_1 = x_2x_3.
\end{equation} 
Write $x_1 = x$. Then $x_2 = xc$ for some $c \in \langle S \rangle^{'}$,
with $c>1$. Moreover, since $x_1x_3 = x_2^2$, we obtain $x_3 = cxc = xc^xc$. Hence, by \eqref{eq31}, 
$cxcx = xccxc$, so  $c^xcx = c^2xc$ and $c^{x^2}c^x = (c^x)^2c$.  Thus 
$$
[c, c^x]^x = [c^x, c^{x^2}] = [c^x, (c^x)^2c(c^x)^{-1}] = [c^x, c]^{(c^x)^{-1}},
$$
yielding $[c, c^x]^{xc^x} = [c^x, c]= [c, c^x]^{-1}$. Since $G$ is an
ordered group, it follows that $[c, c^x] = 1$. Thus $c^{x^{2}} = cc^x$ 
and (b) holds.
\medskip

\noindent Case 2:  $x_1x_3 = x_2x_1$ and $x_3x_1 \in T^2$. 

In this case $x_3=x_2^{x_1}$ and $x_3x_1 \in \{x_1x_2, x_2^2\}$.  
If $x_3x_1 = x_1x_2$, then 
$x_2 = x_3^{x_1}=x_2^{x_1^2}$, so $[x_1^2, x_2] = 1$ and hence 
$[x_1, x_2] = 1$, a contradiction.  So we may suppose that 
\begin{equation}
\label{eq655} 
x_3x_1 = x_2^2.
\end{equation} 
Hence $(x_3x_1)^{-1}x_1x_3=x_2^{-2}x_2x_1=x_2^{-1}x_1$ and $x_2=x_1c$
for some $c \in \langle S \rangle^{'}$ with $c>1$. Moreover, $x_3=x_2^{x_1}
= x_1c^{x_1}$. Write $x_1 = x$; then $x_2 = xc$ and $x_3 = xc^x$.
Moreover $xc^xx = xcxc$, by \eqref{eq655}, hence $c^{x^2} = c^xc$, 
and arguing as before $[c, c^x] = 1$ and (a) holds.
\medskip

\noindent Case 3:  $x_1x_3 = x_2x_1$ and $x_3x_1 \notin T^2$. 

In this case 
$$  
x_1x_3 = x_2x_1, \quad \text{ and } \quad x_3x_1 = x_2x_3.
$$ 
Write $x_1 = x$. Then  $x_2 = xc$ 
for some $c \in \langle S \rangle^{'}$ with $c>1$, $x_3 = xc^x$, 
and $xc^xx = xc xc^x$. Then $c^{x^2}  = c^xc^x$, hence $c^x = c^2$ and (c) holds.
\medskip

We argue similarly if $x_3x_1 \leq x_1x_3$. In this case 
$x_3x_1 \notin \{x_2x_3,  x_3x_2, x_3^2\}$. 
If $x_3x_1 = x_2^2$, then as before $x_1x_3\neq x_3x_1$ and $x_1x_3 = x_3x_2$.
Setting $x_1 = x^{-1}$, $x_2 = x^{-1}c$ for some $c \in \langle S \rangle^{'}$
with $c>1$,  
we obtain that $x_3 = x^{-1}cc^x$ and we see that  $cc^x = c^xc$ and (b) holds. 
If $x_3x_1 = x_1x_2$ and $x_1x_3 = x_3x_2$, then with $x_1 = x^{-1}$ 
we obtain that (c) holds.
If $x_3x_1 = x_1x_2$ and $x_1x_3 = x_2^2$, then with $x_1 = x^{-1}$, 
$x_2 =  x^{-1}c$ for some $c \in \langle S \rangle ^{'}$ with $c>1$, 
we obtain $x_3 = x^{-1}c^x$ and we see that (a) holds.
\end{proof}


\section{The structure of $\langle S \rangle$ if $|S^2| = 3|S|-2$: two general lemmas}
\label{sec4lemgen}

In this section we present two general lemmas which will be useful in the 
inductive process entering the proof of Theorem \ref{3k-2general}, 
as well as in the study of the special case $|S|=4$ (in the next section).

\begin{Lemma}
\label{lemgen2-3}
Let $G$ be an ordered group. Let $S$ be a finite subset of $G$ with 
at least two elements. 
Let either $m= \max S$ or $m= \min S$ and $T= S \setminus \{ m \}$.
Then either $\langle S \rangle$ is a 2-generated abelian group, or $|T^2| \leq |S^2|-3$.
\end{Lemma}

\begin{proof} 
Write $k=|S| \geq 2$ and let $S = \{x_1, x_2, \cdots, x_k\}$, 
with $x_1 < x_2 < \cdots < x_k$. Moreover, suppose that 
$T = \{x_1, \cdots, x_{k-1}\}$. 

Obviously $x_k^2, x_{k-1}x_k, x_kx_{k-1} \notin T^2$, because of the ordering. 
If $x_kx_{k-1} \not= x_{k-1}x_k$, then the result holds immediately. 
Therefore suppose that $x_kx_{k-1} = x_{k-1}x_k$. 

If $x_ix_k, x_kx_i \in T^2$ for each $i < k-1$, then, by Lemma \ref{lemma1}, 
$T \subseteq \langle x_{k-1}, x_k \rangle$. Thus 
$\langle S \rangle = \langle x_{k-1}, x_k \rangle$ and 
hence $\langle S \rangle$ is 2-generated and abelian, as required. 

If there exists $j < k-1$ such that either $x_jx_{k} \notin T^2$ 
or $x_{k}x_j \notin T^2$, then 
$|T^2| \leq |S^2|-3$, since $x_jx_{k}$ and $ x_{k}x_j$ are both 
less than $x_{k-1}x_k = x_kx_{k-1}$, as required.

If $T=\{x_2,\dots,x_k\}$, then the result follows from the previous arguments
by reversing the ordering in $G$.
\end{proof}

Now, we study the case when $S = T \dot\cup \{y\}$, where $\langle T \rangle$ is abelian.

\begin{Proposition}
\label{prop3}
Let $G$ be an ordered group and let $T$ be a finite subset of $G$ such that $|T| \geq 3$ and $\langle T \rangle$ is abelian. 
Let $y \in G \setminus T$. Define $S = T \dot\cup \{y\}$ and assume that $|S^2| = 3|S|-2$. 
Then either $\langle S \rangle$ is abelian, or there are elements $a,c\in G$ such that $S$ is of the form
$$
S = \{a, ac, \dots, ac^{k-2}, y \}
$$ 
where $k=|S|$ and one of the following holds: 
\begin{itemize}
\item[(a)] $[a, y] = c$ or  $[y,a] = c, [c,y] = [c, a] = 1$, 
\item[(b)] $ [a, y] = c$ or $[a,y] = 1, [c, a] = 1, (c^2)^y = c, |S| = 4$
\item[(c)] $[a, y] = 1$ or  $[y,a] = c^2, [c, a] = 1, c^y = c^2, |S| = 4.$
\end{itemize}
In particular, $\langle S \rangle$ is either abelian or young of type (i) or (ii).
If it is  young of type (ii), then $|S|=4$.
\end{Proposition}

\begin{proof}
If $y \in C_G(T)$, then $\langle S \rangle$ is abelian, as required. 
So we may assume that $y \notin C_G(T)$.
Then $|yT \cup Ty| \geq k$ by Proposition 2.4 of \cite{FHLM}. 

We have
$$
y^2 \not\in T^2			
$$
since otherwise,
for any $t \in T$,  $[y^2, t] = 1$ and hence $[y,t] = 1$, thus $y \in C_G(T)$, 
which is not the case. 

Similarly, 
$$
(yT \cup Ty)\cap T^2 = \emptyset.
$$

Hence 
$$
S^2 = T^2 \dot\cup (yT \cup Ty) \dot\cup \{y^2\},
$$ 
and therefore 
$$
|T^2| \leqslant |S^2|-k-1 = 3k-2-k-1 = 2k-3 = 2|T|-1.
$$ 
But by Theorem 1.1 of \cite{FHLM}, it is also true that $|T^2| \geq 2|T|-1$, 
hence $|T^2| = 2|T|-1$, and 
by Corollary 1.4 of \cite{FHLM}, $T$ is of the form 
$$
T = \{a, ac, ac^2, \cdots, ac^{k-2}\},
$$ 
where $ac = ca$, and we may assume that $c >1$. Moreover $|yT \cup Ty| = k$. 
We consider now three cases.
\medskip

\noindent Case 1: $ya < ay$. 

In this case 
$$
yT \cup Ty = \{ya, ay, acy, ac^2y, \cdots, ac^{k-2}y\},
$$ 
with $ya < ay < acy < ac^2y< \cdots < ac^{k-2}y$. 
Now consider the elements $yac, yac^2 \in yT \cup Ty$. Since $yac < yac^2<
yac^3< \cdots <yac^{k-2}$, $yac$ is less than $k-2-1$ elements of $yT \cup Ty$
and $yac^2$ is less than $k-4$ elements of $yT \cup Ty$. 
Thus either $yac = ay$ or $yac = acy$. 

If $yac = acy$, then the only possibility is that $yac^2 = ac^2y$. 
But then $1 = [y, ac] = [y, ac^2]$, yielding $[y, c] = 1 = [y,a]$, a contradiction. 

Now suppose that $yac = ay$, which implies $[a,y] = c $. In this case
either $yac^2 = acy$ or $yac^2 = ac^2y$. 

If $yac^2 = acy$, 
then $ayc = acy$ and $[c,y] = 1$. Thus the derived group 
$\langle S \rangle ^{\prime} = \langle c \rangle \leq Z(\langle S \rangle)$, 
$\langle S \rangle$ is of class 2, and (a) holds.
 
If $yac^2 = ac^2y$, then $ayc = ac^2y$ and $c=(c^2)^y$. 
Moreover, in this case 
$k = 4$. Indeed, if $k>4$, then $yac^3 \in yT \cup Ty$ and 
the only possibility is $yac^3 = ac^3y$. But then  $yac^3 =ayc^2=ac^3y$ 
and $c^2=(c^3)^y=(c^2)^yc^y=cc^y$, yielding $cy = yc$. Thus $yac^2 = ayc^2$,
in contradiction to our assymption that $[a,y]=c$ and $c>1$. 
Therefore $k=4$ and (b) holds.
\medskip

\noindent Case 2: $ay < ya$. 

We argue similarly and obtain that either $[y, a] = c$,  $[c,y] = 1$ and $(a)$ holds, or $[y,a] = c^2$, $c^y = c^2$, 
$|S| = 4$ and (c) holds. 
\medskip

\noindent Case 3: $ay = ya$. 

We have $yac \neq acy$, since otherwise
$[y,a] = 1 = [y,c]$ and $y \in C_G(T)$, a contradiction. 

Assume first $yac < acy$. 
Then 
$$
yT \cup Ty = \{ya = ay, yac, acy, ac^2y, \cdots, ac^{k-2}y\},
$$ 
with $ya = ay < yac < acy < ac^2y<\cdots < ac^{k-2}y$. Consider the element
$yac^2$. Then $yac^2 \in yT \cup Ty$ and
$yac < yac^2< yac^3 < \cdots <yac^{k-2}$, so, as before, we may conclude 
that either  $yac^2 = acy$ or $yac^2 = ac^2y$.
 
If $yac^2 = acy$, then $ya = ay$ implies that $yc^2 = cy$. Thus $c^y = c^2$. 
Moreover, in this case 
$|S| = 4$, since otherwise $yac^3 \in T^2$ and either $yac^3 = ac^2y$, $yc^3 = c^2y$ 
yielding the contradiction $c^3 = (c^2)^y = c^4$, 
or $yac^3 = c^3ay$ yielding the contradiction $cy = yc$.  Therefore (c) holds. 

If $yac^2 = ac^2y$, then $ayc^2 = ac^2y$ and $[c^2, y] = 1$. But then $[c,y] = 1$, 
again a contradiction.

If $acy < yac$, then arguing similarly we obtain $(c^2)^y = c$, $|S| = 4$  and (b) holds.
\medskip

We now prove the `in particular' part of our statement.

If (a) holds, then $\langle S \rangle$ is young of type (i). 

If (b) holds and $[a,y] = 1$, then  $\langle S \rangle$ is young of type (ii). 

If (b)  holds and $[a,y] = c$, then $[ac^2,y] = cc^{-1} = 1$, $(c^2)^y = c$, $[ac^2, c] = 1$  and  again $\langle S \rangle$ is young of type (ii). 

If (c) holds, and $[a,y] = 1$,  we have $\langle S \rangle = \langle a \rangle \times \langle c, y \rangle$, with $c^y = c^2$, and again  $\langle S \rangle$ is young of type (ii).

Finally, if  (c)  holds and $[y,a] = c^2$, then $[y, c] = c^{-1}$, thus $[y, ac^2] = 1$ and  again $\langle S \rangle = \langle ac^2 \rangle \times \langle c, y \rangle$ 
in a young groups of type (ii). 
\end{proof}

\section{The structure of $\langle S \rangle$ if $|S^2| = 3|S|-2$: cardinality 4}
\label{sec4card4}

In this section, we study the special case of Theorem \ref{3k-2general} when $S$ has cardinality 4 and 
prove several lemmas.

\begin{Lemma} 
\label{lemma2v}
Let $G$ be an ordered group. Let $x_1, x_2, x_3,x_4$ be elements of $G$,
such that $x_1 < x_2 < x_3<x_4$ and
let $S = \{x_1, x_2, x_3,x_4\}$. 
Suppose that $|S^2| = 10=3|S|-2$. 
If $x_3 \in Z (\langle x_2, x_3, x_4 \rangle)$, then either 
$\langle x_2, x_3, x_4 \rangle$ or $\langle x_1, x_2, x_3 \rangle$ is abelian.
\end{Lemma}

\begin{proof}
We have $x_2x_3 = x_3x_2$ and  $x_3x_4 = x_4x_3$.

If $x_2x_4 = x_4x_2$, then the result holds. 

We may therefore restrict ourselves to the case where $x_2x_4 \not= x_4x_2$,
and assume that $x_2x_4 < x_4x_2$ 
(we argue similarly in the symmetric case). 

Let $T=\{x_1,x_2,x_3\}$. We first notice that
$$
x_4x_2 \notin T^2.
$$
Indeed, if $x_4x_2 \in T^2$, then the only possibility is 
$x_4x_2 = x_3^2$. But then  $x_2x_4 = x_4x_2$,
a contradiction.

Obviously $x_2x_4 \notin \{x_4x_2, x_3x_4, x_4^2\}$. If also
$x_2x_4 \notin T^2$, then 
$$
|T^2| \leq 10-4 = 6=3|T|-3,
$$ 
and by Theorem 1.3 of \cite{FHLM}
$T$ is abelian, as required.

Thus we may assume that $x_2x_4 \in T^2$ and it follows that 
$$
x_2x_4 \in \{x_3x_1, x_3x_2, x_3^2\}.
$$ 
If $x_2x_4 = x_3x_2 = x_2x_3$, we obtain the contradiction $x_4 = x_3$. 
If $x_2x_4 = x_3x_1$, then $x_1 \in \langle x_2, x_3, x_4\rangle$, 
which implies that $x_1x_3 = x_3x_1$. But then $x_2x_4 = x_1x_3$, in contradiction
to the ordering in $S$. 
Finally, if $x_2x_4 = x_3^2$, then $x_2x_4 = x_4x_2$, again
a contradiction. 
\end{proof}

\begin{Lemma} 
\label{lem444}
Let $G$ be an ordered group. Let $x_1, x_2, x_3,x_4$ be elements of $G$
such that $x_1 < x_2 < x_3<x_4$ and
let $S = \{x_1, x_2, x_3,x_4\}$.
Suppose that $|S^2| = 10=3|S|-2$.
If $x_2x_3=x_3x_2$, then either $\langle x_2,x_3,x_4\rangle$ or 
$\langle x_1,x_2,x_3\rangle$ is abelian.
\end{Lemma}

\begin{proof} 
If $x_3x_4 = x_4x_3$, then the result follows from 
Lemma \ref{lemma2v}. So we may assume that 
$$
x_3x_4 \not = x_4x_3.
$$ 

Write $T = \{x_1, x_2, x_3\}$. If $|T^2| < 7$, then 
Theorem 1.3 of \cite{FHLM} implies that
$\{x_1, x_2, x_3\}$ is abelian, as required. Hence we may assume that
$|T^2| = 7$. Then 
$$
S^2 = T^2 \dot\cup \{x_3x_4, x_4x_3, x_4^2\}.
$$ 

We assume now that
$$
x_2x_4 \leq x_4x_2,
$$ 
the symmetric case being similar.

We first notice that
$$
x_4x_2 \notin T^2,
$$
since otherwise $x_4x_2 \in T^2$ and the only possibility is $x_4x_2 = x_3^2$. 
But then $x_3x_4=x_4x_3$, a contradiction. 
Hence 
$$ 
x_4x_2 = x_3x_4.
$$
and in particular $x_2x_4 \not= x_4x_2$. Moreover, $x_2x_4 \not= x_4x_3$ 
since $x_2x_4 < x_4x_2 < x_4x_3$. 
Hence $x_2x_4 \in T^2$ and $x_2x_4 \in \{x_3x_1, x_3x_2, x_3^2\}$. 

If $x_2x_4 = x_3x_2 = x_2x_3$, then $x_3 = x_4$, a contradiction. 

If $x_2x_4 = x_3^2$, then $x_2x_4=x_4x_2 $, a contradiction. 

Hence the only possibility is
\begin{equation}
\label{eq2811}
x_2x_4 = x_3x_1,
\end{equation}
which implies that
\begin{equation}
\label{eq2812}
x_1x_4 < x_2x_4 = x_3x_1 < x_4x_1.
\end{equation}
If $x_4x_1 \notin T^2$, then $x_4x_1 = x_3x_4 = x_4x_2$, a contradiction. Thus $x_4x_1 \in T^2$ and we have 
$$
x_4x_1 \in \{ x_1x_2, x_2^2, x_2x_3 = x_3x_2, x_1x_3, x_3^2\}.
$$ 
Because of \eqref{eq2812}, the only possibility is
\begin{equation}
\label{eq2813}
x_4x_1 = x_3^2.
\end{equation}
Now, from \eqref{eq2811} we get $x_1x_4^{-1} = x_3^{-1}x_2 \in C_G(x_3)$, and 
by \eqref{eq2813} also 
$x_4x_1 = x_3^2 \in C_G(x_3)$. Thus $x_1^2 \in C_G(x_3)$ and 
$x_1x_3 = x_3x_1=x_2x_4$, 
a contradiction.
\end{proof}

\begin{Lemma} 
\label{lem44young}
Let $G$ be an ordered group. Let $x_1, x_2, x_3,x_4$ be elements of $G$,
such that $x_1 < x_2 < x_3<x_4$ and
let $S = \{x_1, x_2, x_3,x_4\}$.
Suppose that $|S^2| = 10=3|S|-2$.

If $x_1x_2= x_2x_1$ and $x_3x_4 = x_4x_3$, then $\langle S \rangle$ is either abelian or young of type (i), (ii) or (iii).
\end{Lemma}

\begin{proof}

Assume that $\langle S \rangle$ is non-abelian and let us prove that it 
is young of type (i), (ii) or (iii). Let $T=\{x_1, x_2, x_3\}$.

By Lemma \ref{lemgen2-3}, $|T^2| \leq 7$. If $\langle T\rangle$
is abelian, then by Proposition  \ref{prop3}  $\langle S \rangle$ 
is young of type (i) or (ii), as required. 
So assume  that  $\langle T\rangle$ is non-abelian. Then 
Theorem 1.3 of \cite{FHLM} implies that $|T^2|=7$ and Proposition \ref{prop3zs} applies. 

Suppose that $x_1 \in Z(\langle T\rangle)$. Since $|T^2|=7=|S^2|-3$,
it is impossible that $x_1x_4, x_2x_4, x_3x_4, x_4^2 \notin T^2$. Hence
$x_4 \in \langle T \rangle$, which implies that $x_1x_4=x_4x_1$. Thus
$ \langle x_1,x_3,x_4\rangle$ is abelian and by Proposition  \ref{prop3}  $\langle
S \rangle$
is young of type (i) or (ii), as required.

If $x_2$ or $x_3 \in Z(\langle x_1, x_2, x_3 \rangle)$, then $x_2x_3 = x_3x_2$ 
and the result follows by Lemma \ref{lem444}
and Proposition \ref{prop3}. 

So we may assume that $Z(\langle x_1, x_2, x_3 \rangle) \cap \{x_1, x_2, x_3 \} 
= \emptyset$.
Similarly, we may assume that  $\langle  x_2, x_3, x_4\rangle$ is
non-abelian, $|\{x_2, x_3, x_4\}^2|=7$ and 
$Z(\langle x_2, x_3, x_4 \rangle) \cap \{x_2, x_3, x_4 \} = \emptyset$. 
Hence Proposition \ref{prop3zs} implies that
$$
\{x_1, x_2, x_3 \} = \{a, a^b, b\}\quad \text{ and }\quad  \{x_2, x_3, x_4 \} 
= \{c, c^d, d \},
$$ 
for some $a,b,c,d \in G$ satisfying $aa^b = a^ba$, $cc^d = c^dc$. Therefore 
$$
[[a,b], a] = [[a^b, b], a^b] = 1\quad \text{and} \quad   
[[c, d], c] = [[c^d, d], c^d] = 1.
$$	
 Since $[x_1, x_3] \not= 1$ and $[x_1, x_2] = 1$, we have $x_3 = b$, 
$x_2 \in \{a, a^b\}$ and thus $[ [x_2, x_3], x_2] = 1$. 
Similarly, since $[x_2, x_4] \not = 1$ and $[x_3, x_4] = 1$, we have 
$x_2 = d$, $x_3 \in \{c, c^d\}$ and thus $[[x_3, x_2], x_3] = 1$. 
It follows that $\langle x_2, x_3 \rangle$ is nilpotent of class 2. 
Moreover, $a, b \in \langle x_2, x_3 \rangle$ and 
$c, d \in \langle x_2, x_3 \rangle$, so 
$\langle S \rangle = \langle x_2, x_3\rangle$ is 2-generated and nilpotent 
of class 2. Therefore $\langle S \rangle$ 
is young of type (i), as required.
\end{proof} 

\begin{Lemma} 
\label{lem6}
Let $G$ be an ordered group. Let $x_1, x_2, x_3,x_4$ be elements of $G$,
such that $x_1 < x_2 < x_3<x_4$ and
let $S = \{x_1, x_2, x_3,x_4\}$.
Suppose that $|S^2| = 10=3|S|-2$ and $\langle S \rangle$ is non-abelian.
 
If $x_1x_2 = x_2x_1$, then $\langle S \rangle$ is young.		
\end{Lemma}

\begin{proof}
If $x_3x_4 = x_4x_3$, then the result follows from 
Lemma \ref{lem44young}. 
Thus we may suppose that $x_3x_4 \not = x_4x_3$.  

If $x_2x_3 = x_3x_2$, then  
the result follows by Lemma \ref{lem444} and Proposition \ref{prop3}.
Thus we may also assume that $x_2x_3 \not= x_3x_2$.
Moreover, Lemma \ref{lemgen2-3} and Theorem 1.3 in  \cite{FHLM} imply that 
$|\{x_2, x_3, x_4\}^2|=7$.  Therefore 
$\langle x_2, x_3, x_4 \rangle$ satisfies the hypotheses of 
Proposition \ref{prop5}
and hence it is young. 
 
Since $|\{x_2, x_3, x_4\}^2|=7$ and $x_1^2\notin \{x_2, x_3, x_4\}^2$,
it follows that
$$x_1x_2, x_1x_3, x_1x_4 \notin \{x_2, x_3, x_4\}^2$$
is impossible. Hence $x_1 \in \langle x_2, x_3, x_4 \rangle$ and 
$\langle S \rangle = \langle x_2, x_3, x_4 \rangle$ 
is young, as required.
\end{proof}


\section{The structure of $\langle S \rangle$ if $|S^2| = 3|S|-2$: 
proof of Theorem \ref{3k-2general}}
\label{sec4}

Now we can prove Theorem \ref{3k-2general}.

\begin{proof}
Write $S = \{x_1, x_2, \cdots, x_{k-1},  x_k\}$, $x_1 < x_2 < \cdots < x_k$, and define
$$
T = \{x_1, x_2, \cdots, x_{k-1}\},
$$
and
$$
V = \{x_2, \cdots, x_{k-1}, x_k \}.
$$ 
We argue by induction on $k$. 

We start with the basic case $k = 4$. By Lemma \ref{lemgen2-3}, 
either $\langle S \rangle$ is abelian, as required, or $|T^2| \leq 7$. 
Similarly, by considering the order opposite to $<$, we may suppose 
that $|V^2| \leq 7$. 

If either $T$ or $V$ is abelian, then, 
by Proposition \ref{prop3}, $\langle S \rangle$ is either abelian or young of 
type (i) or (ii), as required. So, from now on, we may assume that $T$ and $V$ 
are non-abelian and $|T|=|V|=7$ by Theorem 1.3 in \cite{FHLM}.

If $x_1x_2 = x_2x_1$, then by Lemma \ref{lem6}, $\langle S \rangle$ 
is a young group, as required. So  
we may suppose that $x_1x_2 \not= x_2x_1$ and, by Lemma \ref{lem444},
also $x_2x_3 \not= x_3x_2$. Thus $\langle T \rangle$ satisfies the hypotheses 
of Proposition \ref{prop5} and hence $\langle T \rangle$ is young. 
Moreover, one of the elements $x_1x_4, x_2x_4, x_3x_4 \in T^2$, so 
$\langle S \rangle = \langle T \rangle$ is also young.

Now we move to the inductive step. So suppose that $k > 4$ and 
that the result is true for $k-1$. 
Moreover, suppose that $\langle S \rangle$ is non-abelian. Then 
$$
|T^2| \leq 3k-2-3 = 3(k-1)-2
$$ 
by Lemma \ref{lemgen2-3}. We may assume that the equality holds,
since otherwise  $\langle T \rangle$ is abelian by 
Theorem 1.3 in \cite{FHLM} and the result
follows from Proposition \ref{prop3}. Hence,  by induction, $\langle T \rangle$ 
has the required structure. 

Now consider the $k$ elements $x_1x_k, x_2x_k, \dots, x_{k-1}x_k, x_k^2$. 
If one of them is in $T^2$, then $x_k \in \langle T \rangle$,  
hence $\langle S \rangle = \langle T \rangle$ has the required structure. 
On the other hand, if $x_1x_k, x_2x_k, \cdots, x_k^2 \notin T^2$, 
then $S^2 \supseteq T^2 \dot\cup \{x_1x_k, \cdots, x_k^2\}$ and, in view of $k>4$, 
$$
|T^2| \leq  |S^2| -k  = 2k-2 \leq 3(k-1)-3 = 3 |T|-3.
$$ 
Therefore $T$ is abelian by Theorem 1.3 in \cite{FHLM} and the result 
follows from Proposition \ref{prop3}.
\end{proof}


\section{On the structure of $\langle S \rangle$ if $|S^2| \leq 3|S|-2+s$ and 
$|S|$ is large}
\label{sec5}

We start with the following lemma.

\begin{Lemma} 
\label{lem7}
Let $G$ be an ordered group. Suppose that $a,b,c \in G$ such that $[a,c] = 1$
and let  $T$ be a subset of $G$ satisfying 
$T \subseteq \{a, ac, ac^2, \cdots, ac^h\}$ for some positive 
integer $h$. Moreover, let $b \in G \setminus T$  
such that $|Tb \cap bT| \geq 2$. Then $\langle a, b, c \rangle$ is metabelian. 
Moreover, if $G$ is nilpotent,
then $\langle a, b, c \rangle$ is nilpotent of class at most $2$ .
\end{Lemma}

\begin{proof}
Since $|Tb \cap bT| \geq 2$, there exist $r \not= l$, $s, t \in \mathbb{Z}$ such that $(ac^r)^b = ac^s$, $(ac^l)^b = ac^t$. 
Suppose , without loss of generality, that  $r < l$. Then 
$$
(ac^rc^{l-r})^b = ac^s(c^{l-r})^b= ac^t,
$$ 
which implies that $(c^{l-r})^b = c^{t-s}$. Write $l-r = m$, $t-s = n$. 
Then $(c^m)^b = c^n$, so $[(c^m)^b, c] = 1$ and $[c^b, c] = 1$. 

Now we claim that the subgroup $C = \langle c^{b^j} \ | \ j \in \mathbb{Z} \rangle $ 
is abelian. 

Obviously it suffices to prove that $[c, c^{b^j}] = 1$ for any integer $j$. 
From $(c^m)^b = c^n$ we get easily by induction that $(c^{m^i})^{b^i} = c^{n^i}$ 
for any positive integer $i$. 

Indeed, suppose that $(c^{m^{i-1}})^{b^{i-1}} = c^{n^{i-1}}$. Then we have 
$$
(c^{m^i})^{b^i} = (((c^{m^{i-1}})^{b^{i-1}})^m)^b = ((c^{n^{i-1}})^m)^b 
= ((c^{m})^b)^{n^{i-1}} = c^{n^i}.
$$ 
Therefore $[(c^{b^i})^{m^i}, c] = 1$ for each positive integer $i$,  
and hence $[c^{b^i}, c] = 1$. 
This result  also implies that 
$[c, c^{b^{-i}}] = 1$ for each positive integer $i$. 
Thus $[c, c^{b^j}] = 1$ for every integer $j$ and the claim follows. 

Obviously $b \in N_G(C)$. We claim that also $a \in C_G(C)$. We need only to show
that if $v$ is an integer, then $a^{b^v}\in C_G(c)$. 

We show first that 
$a^b\in  C_G(c)$. Indeed, since $a\in  C_G(c)$ and $(ac^r)^b = ac^s$,
it follows that 
$$a^b(c^b)^r= ac^s=(ac^s)^c=(a^b(c^b)^r)^c=(a^b)^c(c^b)^r,$$
which implies that $(a^b)^c=a^b$, as required. 

Suppose, by induction, that 
$a^{b^v} \in C_G(c)$ for some positive integer $v$.
Since $(ac^r)^{b^{v+1}}=((ac^r)^b)^{b^v}=(ac^s)^{b^v}$, we have
$$a^{b^{v+1}}(c^r)^{b^{v+1}}=(ac^s)^{b^v}=((ac^s)^{b^v})^c
=(a^{b^{v+1}}(c^r)^{b^{v+1}})^c=(a^{b^{v+1}})^c(c^r)^{b^{v+1}}.$$
Hence also $a^{b^{v+1}} \in C_G(c)$. It follows that 
$a^{b^v} \in C_G(c)$ for each positive integer $v$.

Similarly, from $(ac^s)^{b^{-1}} = ac^r$, 
we get that $a^{b^{-1}} \in C_G(c)$ and 
by induction $a^{b^{-v}} \in C_G(c)$ for each positive integer $v$.
 
Hence $a \in C_G(C) \subseteq N_G(C)$. Thus $C$ is normal 
in $\langle a, b, c \rangle$, 
and obviously $\langle a, b, c \rangle /C$ is abelian. Hence 
$\langle a, b, c \rangle$ is metabelian.

If  $G$ is a torsion-free 
nilpotent group, then  $(c^m)^b = c^n$ implies
that $c^b = c$. 
In fact, we can argue by induction on the nilpotency class $d$ of $G$. 
The result is obvious if $d = 1$, 
that is if $G$ is abelian. By induction we have $c^bZ(G) = cZ(G)$, 
since $G/Z(G)$ is torsion-free of class $d-1$ 
(see for example 5.2.19 of \cite{Robinson}). Thus $c^b = cz$
for some $z \in Z(G)$ 
and we have $c^n = (c^b)^m = (cz)^m = c^mz^m$, which implies that
$c^{n-m} \in Z(G)$. If $n-m \not = 0$, then $c \in Z(G)$, and obviously $c^b = c$ 
in this case. If $m = n$, 
then $[c^m , b] = 1$ and $[c,b] = 1$, since $G$ is an ordered group.
Hence $C\subseteq Z(\langle a, b, c \rangle)$ and 
$\langle a, b, c \rangle$ is nilpotent of class at most $2$ , as required.
\end{proof}

Lemma \ref{lem7} has the following useful Corollary.

\begin{cor} 
\label{cororo}
Let $G$ be an ordered group and let $S$ be a finite subset of $G$ of finite size $> 3$. 
Let $m= \max S$ and $T= S \setminus \{m \}$.  
Suppose that $|S^2| = 3|S|+b$ for some $b \leq |S|-6$.
If $m \notin \langle T \rangle$ and $m \notin C_G(T)$, then 
$\langle S \rangle$ is metabelian and if $G$ is nilpotent, 
then $\langle S \rangle$ is nilpotent of class at most $2$.
\end{cor}

\begin{proof} 
Set $|S|=k$ and $m=x_k$. Since $x_k\notin \langle T \rangle$ and $x_k=\max S$, 
we have the partition
$$
S^2 = T^2 \dot\cup (x_kT \cup Tx_k) \dot\cup \{x_k^2\}.
$$
Moreover $|x_kT \cup Tx_k| \geq k$ by Proposition 2.4 of \cite {FHLM}. 
Hence 
$$
3k+b = |T^2|+1+|x_kT \cup Tx_k| \geq |T^2|+1+k,
$$ 
and 
$$
|T^2| \leq 2k+b-1 \leq 2k+k-6-1 = 3(k-1)-4.
$$ 
Thus, by Corollary 1.4 of \cite{FHLM}, there exist elements $a, c$ in $G$, 
with $ac = ca$, 
such that $T \subseteq \{a, ac, \cdots, ac^h\}$ for some positive integer $h$. 

Furthermore, $|T^2|\geq 2|T|-1=2k-3$ by Theorem 1.1 in  \cite{FHLM}. Therefore
$$|x_kT \cup Tx_k|=|S^2|-|T^2|-1\leq (4k-6)-(2k-3)-1=2k-4$$
and
$$|x_kT \cap Tx_k|=2|T|-|x_kT \cup Tx_k|\geq 2(k-1)-(2k-4)=2.$$ 
Thus Lemma \ref{lem7} applies. 
Hence $\langle S \rangle \subseteq \langle a, c, x_k \rangle$ is metabelian and
and if $G$ is nilpotent,
then $\langle S \rangle$ is nilpotent of class at most $2$, as required.
\end{proof}

Finally, in order to prove Theorem \ref{3k-2+s}, we shall use the folllowing 
easy consequence of Lemma \ref{lemgen2-3}.

\begin{Lemma} 
\label{lem88}
Let $G$ be an ordered group and let $x_1 < \cdots < x_k$ be elements of $G$. 
Let $S = \{x_1, \cdots, x_k\}$, with $k\geq2$ and suppose that $|S^2| \leq 3|S|+b$
for some integer $b$. 
Then for any integer $i$, $1\leq i < k$, either $\langle x_1, \cdots, x_{i+1} \rangle$ 
is abelian, or $|\{x_1, \cdots, x_i\}^2| \leq 3i+b$.
\end{Lemma}

\begin{proof}
Notice that if $b\leq -3$, then $\langle S \rangle$ is abelian by Theorem 1.3
in \cite{FHLM}. Thefore we may assume that $b\geq -2$.

Our proof is by induction from $i=k-1$ down to $k=1$. If $i=k-1$, then
$\langle x_1, \cdots, x_{i+1} \rangle= \langle S \rangle$. Hence, by 
Lemma \ref{lemgen2-3}, either $ \langle S \rangle$ is abelian, or
$$|\{x_1, \cdots, x_i\}^2|\leq |S^2|-3\leq 3k-3+b=3(k-1)+b,$$
as required.

So assume that $1\leq j<k-1$ and that the Lemma holds for $i=j+1$. Moreover,
suppose that $\langle x_1, \cdots, x_{j+1} \rangle$ is non-abelian. Then
$\langle x_1, \cdots, x_{j+2} \rangle$ is nonabelian and by our assumptions
$|\{x_1, \cdots, x_{j+1}\}^2|\leq 3(j+1)+b$. 
It follows then by  Lemma \ref{lemgen2-3} that 
$$|\{x_1, \cdots, x_j\}^2\leq |\{x_1, \cdots, x_{j+1}\}^2|-3\leq 3j+b,$$
as required.
Therefore the Lemma holds for $1\leq i<k$, as required.
\end{proof}

We can now prove Theorem \ref{3k-2+s}.

\begin{proof} [Proof of Theorem \ref{3k-2+s}]
Write $S = \{x_1, \cdots, x_k\}$,
with $x_1 < \cdots < x_k$, $T = \{x_1, \cdots, x_t\}$, $V = 
\{x_{t+1}, \cdots, x_k\}$, with $|T| = |V| = \frac k2$
if $k$ is even, and $|T| = \frac{k-1}2$, $|V| = \frac{k+1}2$, if $k$ is odd.
Then  $t = |T| \geq 2^{s+1}$ and $v = |V|\geq 2^{s+1}$.

We claim that either $\langle T \rangle$ or $\langle V \rangle$ is
metabelian
(and nilpotent of class at most 2 if $G$ is nilpotent).
We shall prove this claim by induction on $s$.

The claim is obvious if either $\langle T \rangle$ or $\langle V \rangle$ is
abelian. So suppose that both are non-abelian.
Then, by Lemma \ref{lem88},
$|T^2| \leq 3t-2+s$ and also $|V^2| \leq 3v-2+s$, by considering the
ordering
opposite to $<$.

Moreover, because of the ordering,
$T^2 \cap V^2 = \emptyset$ and
$x_tx_{t+1}, x_{t+1}x_t \notin T^2 \cup V^2$.
If $x_tx_{t+1} = x_{t+1}x_t$, then by Lemma \ref{lemma1}
there exists $x_j < x_t$ such that either
$x_jx_{t+1}$ or $x_{t+1}x_j$ is not in $T^2$,
since $\langle T \rangle$ is not abelian. Obviously also
$x_jx_{t+1}, x_{t+1}x_j \notin V^2$.
Therefore, in any case $S^2 \setminus (T^2 \cup V^2) \geq 2$.

Now, if
$$
|T^2| \geq 3t-2+\frac{s+1}2
$$
and
$$
|V^2| \geq 3v-2+\frac{s+1}2,
$$
then
$$
|S^2|\geq 3k-4+s+1+2 = 3k+s-1,
$$
a contradiction.

If $s=1$, then this contradiction implies
that either $|T^2|\leq 3t-2$ or
$|V^2|\leq 3v-2$ and by Corollaries \ref{cor2} and \ref{cor1}
either
$\langle T \rangle$ or $\langle V \rangle$ is metabelian
(and nilpotent of class at most 2 if $G$ is nilpotent), as claimed.

So suppose that $s\geq 2$, and that our claim holds for all smaller
values of $s$. By the above contradiction, we must have either 
$$
|T^2| \leq 3t-2+\frac s2
$$
or
$$
|V^2| \leq 3v-2+\frac s2.
$$
Moreover
$$
t \geq 2^{s+1} \geq 2^{\frac s2+2} 
$$
and similarly $v \geq  2^{\frac s2 +2}$, since $s\geq 2$. Hence, by
induction,
either $\langle T \rangle$ or $\langle V \rangle$ is metabelian
(and nilpotent of class at most 2 if $G$ is nilpotent), as claimed.
The proof of the claim is complete.

Suppose,
without loss of generality, that $\langle T \rangle$ is metabelian 
(and nilpotent of class at most 2, if $G$ is nilpotent). Let 
$X = \{x_1, \cdots, x_j\}\supseteq T$ be a subset of $S$ maximal under
the condition that $\langle X \rangle$ is metabelian  (and nilpotent of
class at most 2, if $G$ is nilpotent). If $X = S$, then we have the result.

So suppose that $X$  is a proper subset of $S$.
Under this assumption we shall reach a contradiction, thus
concluding the proof of the Theorem. Because of the maximality of $X$,
$x_{j+1} \notin \langle X \rangle$ and $x_{j+1} \notin C_G(X)$.
Hence, if we write $W = \{x_1, \cdots, x_{j+1}\}$ and $w = |W|$, then
$|W^2|\leq 3w-2+s$, by Lemma \ref{lem88}. Moreover $s-2\leq w-6$,
since $w\geq t+1\geq 2^{s+1}+1\geq s+4$.
Therefore Corollary \ref{cororo} applies, and $\langle W \rangle$ is
metabelian
(and nilpotent of class at most 2, if $G$ is nilpotent), a contradiction.
\end{proof}

Finally we can prove our final statement. 

\begin{proof}[Proof of Theorem \ref{construction}]
Let $G = \langle a \rangle \times \langle b, c \rangle$, where $\langle a \rangle$ 
is an infinite cyclic group and $\langle b,c \rangle$ is a free group of rank 2. 
Then $G$ is a direct product of two orderable groups and therefore it is an orderable group. 
Let $k$ be an integer $\geq 3$ and define $S = \{a, ac, \cdots, ac^{k-2}, b\}$. 
Write $T = \{a, ac, \cdots, ac^{k-2}\}$. Then $b \notin C_G(T)$, 
so in particular $b \notin \langle T \rangle$. Hence $S^2 = T^2 \dot\cup (bT \cup Tb)
\dot\cup \{b^2\}$. 
We also have $ab = ba$ and $bac^i = abc^i \not= ac^jb$ for any $i \not= j$, 
since $\langle b, c \rangle$ is free.  Hence $|bT \cap Tb| = 1$, which
implies that 
$$
|bT \cup Tb| = k-1+k-1-1 = 2k-3.
$$ 
Since $S^2 = T^2 \dot\cup (bT \cup Tb) \dot\cup \{b^2\}$, it follows that 
$$
|S^2| = 2(k-1)-1+2k-3+1 = 4k-5.
$$ 
Obviously $\langle S \rangle = \langle a,b,c\rangle=G$. In particular,
$\langle S \rangle$ is not soluble.
\end{proof}

\section*{Acknowledgements}

The authors wish to thank Professor Eamon O'Brien (Auckland, New-Zea\-land), 
for his precious suggestions 
concerning the statement of Theorem \ref{3k-2general}.

This work was supported by the ``National Group for Algebraic and Geometric 
Structures, and their Applications" (GNSAGA - INDAM), Italy.

A.P. is supported by the ANR grant Caesar number 12 - BS01 - 0011. He thanks 
his colleagues from the University of Salerno for their hospitality.

\end{document}